\documentclass[preprint,a4paper,12pt]{elsarticle}

\usepackage{lineno}

\usepackage{amsmath}
\usepackage{amsfonts}
\usepackage{amssymb}
\usepackage{amsthm}

\usepackage{bm}
\usepackage{algorithm}
\usepackage{algpseudocode}

\usepackage{graphicx}
\usepackage{subcaption}
\usepackage{multirow} 
\usepackage{float} 

\theoremstyle{definition}
\newtheorem{theorem}{Theorem}[section]

\usepackage{xcolor}
\usepackage[colorlinks=true, allcolors=blue]{hyperref}

\begin{document}
	
\begin{frontmatter}
	
	
	
    \title{Hybrid Joint--Selective Optimization: Reduced-Space
    	Levenberg--Marquardt Refinement of Low-Dimensional Parameters of Interest}
	
	
	
	\author[inst1]{Muhammad Luthfi Shahab}

    \author[inst1]{Gabriella Alfa Indahsari}

	\author[inst1]{Imam Mukhlash}
    
	\author[inst2]{Hadi Susanto}

	\affiliation[inst1]{
    	organization={Department of Mathematics, Institut Teknologi Sepuluh Nopember},
    	city={Surabaya},
    	postcode={60111}, 
    	country={Indonesia}
    }
    
    \affiliation[inst2]{
    	organization={Department of Mathematics, Khalifa University of Science \& Technology},
    	city={Abu Dhabi},
    	postcode={PO Box 127788}, 
    	country={United Arab Emirates}
    }
		
	\begin{abstract}
		This paper introduces a hybrid joint--selective optimization (HJSO)
		framework for large-scale numerical problems in which a small subset of
		trainable quantities is of primary interest. We partition the full
		parameter vector into a high-dimensional remaining block and a
		low-dimensional block of parameters of interest (POIs), perform joint
		first-order optimization over the full parameter set, and then freeze the
		remaining variables while applying a reduced-space Levenberg--Marquardt
		(LM) refinement to the POIs. The method is designed for settings in which
		the POIs are low-dimensional but strongly influence the
		quality of the computed solution, while the full parameter space remains too
		large for full-space second-order methods.

		The framework is evaluated on three representative problems: a matrix
		eigenvalue problem, an inverse Bratu problem solved with a physics-informed
		neural network, and a 100-dimensional nonlinear Black--Scholes problem
		solved with the DeepBSDE method. In each test, HJSO reaches prescribed
		POI-error thresholds faster than the corresponding joint first-order
		baseline and improves the final POI accuracy for the reported solver
		configurations. The contribution is therefore not a universal optimizer, but a
		practical reduced-space strategy for problems with known low-dimensional
		parameters of interest and expensive high-dimensional training variables.
	\end{abstract}
	

\end{frontmatter}


\section{Introduction}
\label{sec:introduction}

Large-scale numerical optimization is central to modern scientific
computing, especially in neural-network-based methods for differential
equations and stochastic control problems \cite{chen2018neural}. Physics-informed neural networks
(PINNs) \cite{raissi2019physics}, DeepBSDE formulations for high-dimensional
parabolic PDEs \cite{han2018solving}, and other data-driven discretizations
often involve a large number of trainable parameters. In these settings,
first-order methods such as gradient descent, steepest descent
\cite{chong2023introduction}, or Adam \cite{kingma2014adam} remain the
default choice because they scale naturally to large models and avoid the
cost of constructing or factorizing full Hessians.

However, many problems are not equally sensitive in all directions of the
parameter space. A small subset of variables may have a disproportionately
large effect on the quantity of interest, even when the total parameter
count is large. This is common in inverse problems, eigenvalue problems,
and stochastic PDE approximations, where the unknowns of primary interest may
be a scalar or a few coefficients rather than the full network parameter
vector itself. In PINNs, for example, imbalanced convergence among loss
components and gradient-flow pathologies are well documented
\cite{wang2021understanding,wang2022and}; in other settings, the final
accuracy of a target quantity can be limited by slow convergence of a
low-dimensional block that influences the solution disproportionately \cite{putri2024deep}.

This observation motivates a structured optimization viewpoint. Rather than
optimizing all parameters with a single generic first-order update, it can
be advantageous to distinguish between the large remaining parameter block
and a small set of quantities that are mathematically or physically central
to the application. This idea echoes classical separable-variable and
block-structured optimization methods, including variable projection
\cite{golub2003separable}, coordinate descent
\cite{nesterov2012efficiency,wright2015coordinate}, and block coordinate
updates \cite{blondel2013block,richtarik2016parallel}. Related ideas have
also been explored in neural-network training, where parameters are grouped
according to architecture or linear/nonlinear structure
\cite{mcloone1998hybrid,cyr2020robust,patel2020block}.

The present work differs in a key respect: the selected parameters are not
chosen solely by network architecture or algebraic structure, but by their
problem-specific role as low-dimensional \emph{parameters of interest} (POIs)
\cite{naderibeni2024learning}. We
consider settings such as an unknown eigenvalue in a spectral problem, an
unknown coefficient pair in an inverse differential equation, and an unknown
initial value in a DeepBSDE formulation. In all of these cases, the POIs are
few in number but they determine the numerical output of primary interest.
This structure is common in computational science and motivates reduced-space
optimization strategies that use second-order-type information only on the
small POI block rather than on the full parameter vector.

Accordingly, the trainable vector is partitioned into a high-dimensional
remaining block and a low-dimensional POI block. A standard joint method
updates both blocks simultaneously but does not exploit their different roles;
the formal partition is introduced in Section~\ref{subsec:hybrid_method}.

Curvature-based optimization provides a potential mechanism for accelerating
POI refinement. In particular, the Levenberg--Marquardt (LM) algorithm
\cite{levenberg1944method,marquardt1963algorithm} is well suited to nonlinear
least-squares problems because it exploits local Gauss--Newton curvature
information without requiring the exact Hessian. The LM algorithm has also
been applied to neural-network training, including its incorporation into
backpropagation for feedforward neural networks
\cite{hagan1994training}. Applying LM to the complete high-dimensional
parameter vector, however, can be computationally demanding because of the
associated Jacobian and linear-algebra costs. This motivates restricting
curvature-based refinement to a low-dimensional parameter subspace.

Hybrid strategies involving first-order and
curvature-based optimization have also been investigated. Berra et
al.~\cite{berra2024combined} combined projected Newton and gradient steps for
box-constrained root-finding problems. Costilla-Enriquez et
al.~\cite{costilla2020combining} combined Newton--Raphson and stochastic
gradient descent for power-flow analysis, switching between the two
algorithms according to convergence behavior. These approaches apply
different optimization algorithms over essentially the same variable space.
In contrast, the present work explicitly partitions the trainable parameters
and restricts curvature-based refinement to a low-dimensional POI block.

This structure motivates the \emph{hybrid joint--selective optimization}
(HJSO) method. Within each outer cycle, all trainable parameters are first
optimized jointly using a first-order method. The remaining parameters are
then fixed, and the LM algorithm selectively refines only the POIs. The
refined POIs are subsequently recombined with the remaining parameters to
initialize the next outer cycle. Thus, the defining feature of HJSO is not
merely the combination of first-order and curvature-based optimization, but
the joint optimization of the complete parameter vector followed by selective
curvature-based refinement of a low-dimensional POI block. Since the POI
dimension is much smaller than that of the remaining parameter space, the
resulting LM subproblem is substantially smaller than curvature-based
optimization of the complete model.

HJSO should be viewed as a reduced-space optimization framework rather than
as a new variant of LM. Its contribution is the systematic placement of a
standard damped Gauss--Newton refinement inside repeated joint optimization,
with the reduced variables chosen according to their mathematical role in the
problem. This interpretation also distinguishes the proposed strategy from
variable projection: HJSO does not require the POI block to be eliminated
exactly, and the remaining variables continue to be updated jointly with the
POIs between selective refinements.

The proposed method is evaluated on three numerical problems. First, the
largest eigenvalue of a $200\times200$ Lehmer matrix is considered, with the
eigenvalue treated as a scalar POI. Second, an inverse Bratu problem is solved
using PINNs \cite{shahab2026physics}, where two unknown coefficients,
$\lambda_1$ and $\lambda_2$, form a two-dimensional POI block. Finally, a
100-dimensional nonlinear Black--Scholes equation with default risk is solved
using DeepBSDE \cite{han2018solving}, where the unknown initial value
$u_0=u(0,\bm X_0)$ is treated as the POI. These examples cover scalar and
vector-valued POIs, deterministic and stochastic formulations, and low- and
high-dimensional numerical problems.

The main contributions of this work are summarized as follows:
\begin{enumerate}
	\item A hybrid joint--selective optimization framework is proposed that
	combines joint first-order optimization of all trainable parameters with
	selective LM refinement of a low-dimensional POI block.
	
	\item The framework accommodates both scalar and vector-valued POIs while
	restricting curvature-based optimization to a substantially smaller
	parameter subspace.
	
	\item The reduced-space formulation is discussed in terms of its
	computational role and its interaction with the joint optimization phase,
	which clarifies when selective refinement is a practical and scalable
	strategy for problems with low-dimensional POIs.
\end{enumerate}

The remainder of this paper is organized as follows. Section~\ref{subsec:hybrid_method}
presents the proposed HJSO formulation and computational algorithm.
Section~\ref{sec4} presents the numerical experiments, and
Section~\ref{sec:discussion} discusses their scope and limitations. Finally,
Section~\ref{sec:conclusion} summarizes the main conclusions and outlines
promising directions for future work.

\section{Hybrid Joint--Selective Optimization (HJSO)}
\label{subsec:hybrid_method}

We now formulate HJSO for the prescribed POI block introduced in
Section~\ref{sec:introduction}.

Let the full trainable parameter vector be partitioned as
\begin{equation}
\bm{\theta}
=
\begin{bmatrix}
\bm{\theta}_{\mathrm r}\\
\bm{\theta}_{\mathrm{poi}}
\end{bmatrix},
\qquad
\bm{\theta}_{\mathrm r}\in\mathbb{R}^{p},
\qquad
\bm{\theta}_{\mathrm{poi}}\in\mathbb{R}^{q},
\qquad
q\ll p,
\label{eq:parameter_partition}
\end{equation}
where $\bm{\theta}_{\mathrm r}$ denotes the high-dimensional remaining
parameter block and $\bm{\theta}_{\mathrm{poi}}$ denotes the low-dimensional
POI block. In neural-network-based problems, for example,
$\bm{\theta}_{\mathrm r}$ may contain the network weights and biases, while
$\bm{\theta}_{\mathrm{poi}}$ may encode model coefficients, eigenvalues, or
other quantities of primary scientific relevance.

In a conventional joint optimization strategy, all trainable parameters are
updated simultaneously. For example, gradient descent updates the complete
parameter vector according to
\begin{equation}
\bm{\theta}^{(m+1)}
=
\bm{\theta}^{(m)}
-
\alpha_m
\nabla_{\bm{\theta}}
\mathcal L
\left(
\bm{\theta}^{(m)}
\right),
\label{eq:standard_gd}
\end{equation}
where $\alpha_m>0$ is the learning rate and $\mathcal L$ is the objective
function. This is often attractive because it is simple and scalable, but it
does not distinguish between globally large parameter blocks and the small
subset that directly controls the output of interest.

The proposed framework, \emph{hybrid joint--selective optimization} (HJSO),
addresses this imbalance by alternating between an outer joint update and a
reduced-space refinement. The method is not tied to a specific optimizer in
principle: the joint phase may use any scalable first-order procedure, while
the selective phase may use a local reduced-space method for the POIs. In the
present implementation, a first-order optimizer is used in the joint phase,
while the Levenberg--Marquardt (LM) algorithm is used in the selective phase.
This combination is chosen because it preserves the scalability of
first-order optimization over the large parameter space while exploiting
curvature information only in the substantially smaller POI subspace.

The key structural idea is therefore to perform a full-space joint update,
freeze the large remaining block, and then refine only the low-dimensional
POI block. The refined POIs are then recombined with the remaining parameters
and used to initialize the next outer cycle. This makes the method particularly
relevant for large-scale scientific problems in which the dominant objective
is not the full parameter vector itself but a small set of physically or
mathematically meaningful quantities derived from it.
At the $k$-th outer cycle, $k=0,1,\ldots,T-1$, let
\begin{equation}
	\bm{\theta}^{(k)}
	=
	\begin{bmatrix}
		\bm{\theta}_{\mathrm r}^{(k)}\\
		\bm{\theta}_{\mathrm{poi}}^{(k)}
	\end{bmatrix}
\end{equation}
denote the current parameter vector. In the joint phase, consider the
optimization problem
\begin{equation}
	\min_{\bm{\theta}}
	\;
	\mathcal L(\bm{\theta}).
\end{equation}
Starting from
\begin{equation}
	\bm{\vartheta}^{[0]}
	=
	\bm{\theta}^{(k)},
	\label{eq:fo_initialization}
\end{equation}
where $\bm{\vartheta}^{[j]}$ denotes the $j$-th inner first-order iterate, the
complete parameter vector is updated according to
\begin{equation}
	\bm{\vartheta}^{[j+1]}
	=
	\operatorname{FOUpdate}
	\left(
	\bm{\vartheta}^{[j]}
	\right),
	\qquad
	j=0,1,\ldots,T_{\mathrm{FO}}-1,
	\label{eq:fo_inner_update}
\end{equation}
where $T_{\mathrm{FO}}$ is the prescribed maximum number of first-order
iterations. The joint phase may terminate earlier according to the stopping
criterion of the selected solver. Its final iterate is denoted by
\begin{equation}
	\bm{\theta}^{(k+\frac12)}
	=
	\begin{bmatrix}
		\bm{\theta}_{\mathrm r}^{(k+\frac12)}\\
		\bm{\theta}_{\mathrm{poi}}^{(k+\frac12)}
	\end{bmatrix}
	= \bm{\vartheta}^{[T_{\mathrm{FO}}]}.
	\label{eq:fo_final_update}
\end{equation}
Hence, both the remaining parameters and the POIs are updated during the joint
phase.

In the selective phase, the remaining parameter block is fixed at
\begin{equation}
	\overline{\bm{\theta}}_{\mathrm r}
	=
	\bm{\theta}_{\mathrm r}^{(k+\frac12)},
	\label{eq:frozen_parameters}
\end{equation}
and only the POIs are further optimized. The reduced problem is written in
nonlinear least-squares form as
\begin{equation}
	\min_{\bm{\theta}_{\mathrm{poi}}}
	\;
	\Phi(\bm{\theta}_{\mathrm{poi}})
	=
	\frac{1}{2}
	\left\|
	\bm r
	\left(
	\overline{\bm{\theta}}_{\mathrm r},
	\bm{\theta}_{\mathrm{poi}}
	\right)
	\right\|_2^2,
	\label{eq:reduced_problem}
\end{equation}
where
\begin{equation}
	\bm r
	\left(
	\overline{\bm{\theta}}_{\mathrm r},
	\bm{\theta}_{\mathrm{poi}}
	\right)
	=
	\begin{bmatrix}
		r_1\left(
		\overline{\bm{\theta}}_{\mathrm r},
		\bm{\theta}_{\mathrm{poi}}
		\right) &
		\cdots &
		r_{N_r}\left(
		\overline{\bm{\theta}}_{\mathrm r},
		\bm{\theta}_{\mathrm{poi}}
		\right)
	\end{bmatrix}^{T}
	\in\mathbb R^{N_r}
\end{equation}
is the residual vector.

The Jacobian of the residual with respect to the POIs is
\begin{equation}
	\bm J_{\mathrm{poi}}
	=
	\frac{
		\partial
		\bm r
		\left(
		\overline{\bm{\theta}}_{\mathrm r},
		\bm{\theta}_{\mathrm{poi}}
		\right)
	}{
		\partial\bm{\theta}_{\mathrm{poi}}
	}
	\in\mathbb{R}^{N_r\times q}.
	\label{eq:poi_jacobian}
\end{equation}
For $q=1$, $\bm J_{\mathrm{poi}}$ consists of a single sensitivity column,
whereas for $q>1$ it contains one column for each POI. LM is particularly
suitable for this reduced nonlinear least-squares problem because it exploits
Gauss--Newton curvature information while introducing a damping term that
improves robustness when
$\bm J_{\mathrm{poi}}^{T}\bm J_{\mathrm{poi}}$ is singular,
rank-deficient, or ill-conditioned \cite{shahab2026physics}.

The selective phase is initialized by
\begin{equation}
	\bm{\theta}_{\mathrm{poi}}^{[0]}
	=
	\bm{\theta}_{\mathrm{poi}}^{(k+\frac12)}.
	\label{eq:selective_initialization}
\end{equation}
At the $j$-th LM refinement, the increment is obtained from
\begin{equation}
	\left[
	\left(\bm J_{\mathrm{poi}}^{[j]}\right)^T
	\bm J_{\mathrm{poi}}^{[j]}
	+
	\mu_j\bm I_q
	\right]
	\Delta\bm{\theta}_{\mathrm{poi}}^{[j]}
	=
	\left(\bm J_{\mathrm{poi}}^{[j]}\right)^T
	\bm r^{[j]},
	\label{eq:lm_system}
\end{equation}
followed by
\begin{equation}
	\bm{\theta}_{\mathrm{poi}}^{[j+1]}
	=
	\bm{\theta}_{\mathrm{poi}}^{[j]}
	-
	\Delta\bm{\theta}_{\mathrm{poi}}^{[j]},
	\qquad
	j=0,1,\ldots,T_{\mathrm{LM}}-1,
	\label{eq:lm_update}
\end{equation}
where
\begin{equation}
	\bm r^{[j]}
	=
	\bm r
	\left(
	\overline{\bm{\theta}}_{\mathrm r},
	\bm{\theta}_{\mathrm{poi}}^{[j]}
	\right),
\end{equation}
$\mu_j>0$ is the LM damping parameter, and
$\bm I_q\in\mathbb R^{q\times q}$ is the identity matrix. The selective phase
may terminate before reaching $T_{\mathrm{LM}}$ iterations according to the
stopping criterion of the LM solver. Depending on the numerical software, its
computational budget may instead be specified by a maximum number of residual
function evaluations.

The damping parameter allows LM to interpolate between Gauss--Newton and a
gradient-like update. For sufficiently small $\mu_j$,
Eq.~\eqref{eq:lm_system} approaches
\begin{equation}
	\bm J_{\mathrm{poi}}^{T}\bm J_{\mathrm{poi}}
	\Delta\bm{\theta}_{\mathrm{poi}}
	=
	\bm J_{\mathrm{poi}}^{T}\bm r,
\end{equation}
which is the Gauss--Newton system. For relatively large $\mu_j$,
\begin{equation}
	\Delta\bm{\theta}_{\mathrm{poi}}
	\approx
	\frac{1}{\mu_j}
	\bm J_{\mathrm{poi}}^{T}\bm r.
\end{equation}
Since
\begin{equation}
	\nabla_{\bm{\theta}_{\mathrm{poi}}}\Phi
	=
	\bm J_{\mathrm{poi}}^{T}\bm r,
\end{equation}
the latter resembles a scaled gradient step. LM therefore combines local
curvature information with damping that improves numerical robustness.

Let $j_k$ denote the final accepted selective iterate in outer cycle $k$;
$j_k=T_{\mathrm{LM}}$ if the iteration limit is reached. After the selective
phase, the refined POIs are recombined with the fixed remaining parameters,
\begin{equation}
	\bm{\theta}^{(k+1)}
	=
	\begin{bmatrix}
		\bm{\theta}_{\mathrm r}^{(k+\frac12)}\\
		\bm{\theta}_{\mathrm{poi}}^{[j_k]}
	\end{bmatrix}.
	\label{eq:reintegrate_parameter}
\end{equation}
The resulting vector initializes the joint phase of the next outer cycle. 
The complete HJSO procedure is summarized
in Algorithm~\ref{alg:hybrid_fo_lm}.

The principal structural advantage of HJSO follows from $q\ll p$.
Curvature-based optimization is restricted to the $q$-dimensional POI
subspace, while the high-dimensional remaining parameter block is handled by
a scalable first-order method. The two phases therefore play complementary
roles: the joint phase allows the complete parameter vector to adapt
simultaneously, whereas the selective phase provides focused refinement of the
POIs. These refined POIs are then returned to the complete parameter vector
and influence the joint optimization in the next outer cycle.

The small size of the linear system does not, by itself, make the selective
phase inexpensive. If the residual contains $N_r$ components, forming a dense
POI Jacobian and its normal matrix requires, respectively,
$O(N_rq)$ storage and $O(N_rq^2)$ arithmetic, in addition to the cost of
residual and sensitivity evaluations. Solving the resulting dense LM system
costs $O(q^3)$. Consequently, selective refinement is most attractive when
$q$ is small and residual/Jacobian evaluations can be reused, differentiated
efficiently, or computed at a cost well below curvature calculations in the
full $(p+q)$-dimensional space.

The following elementary property records the descent mechanism used by the
framework. It is conditional because practical first-order and LM solvers may
terminate inexactly and stochastic objectives may change between batches.

\begin{theorem}[Conditional cycle-wise descent]
Suppose that, during one outer cycle, the joint phase returns
$\bm\theta^{(k+\frac12)}$ satisfying
\begin{equation}
 \mathcal L\!\left(\bm\theta^{(k+\frac12)}\right)
 \leq \mathcal L\!\left(\bm\theta^{(k)}\right).
\end{equation}
Assume further that, with $\bm\theta_{\mathrm r}$ fixed, the joint objective
can be written as
\begin{equation}
 \mathcal L(\overline{\bm\theta}_{\mathrm r},
 \bm\theta_{\mathrm{poi}})
 =c\,\Phi(\bm\theta_{\mathrm{poi}})+C,
 \qquad c>0,
 \label{eq:selective_objective_compatibility}
\end{equation}
where $C$ may depend on the frozen block
$\overline{\bm\theta}_{\mathrm r}$ but is independent of
$\bm\theta_{\mathrm{poi}}$ during the selective phase. If every accepted LM
step does not increase $\Phi$, then the completed HJSO cycle satisfies
\begin{equation}
 \mathcal L\!\left(\bm\theta^{(k+1)}\right)
 \leq \mathcal L\!\left(\bm\theta^{(k)}\right).
\end{equation}
\end{theorem}

\begin{proof}
Let $j_k$ denote the final accepted selective iterate. The LM acceptance
condition requires
$\Phi(\bm\theta_{\mathrm{poi}}^{[j_k]})
\leq\Phi(\bm\theta_{\mathrm{poi}}^{(k+\frac12)})$.
Moreover, the initial and final selective-phase vectors have the same frozen
remaining block $\overline{\bm\theta}_{\mathrm r}$. Since $c>0$, multiplying
this inequality by $c$ and adding the same constant $C$ preserves its
direction. Equation~\eqref{eq:selective_objective_compatibility} and the
assumed joint-phase decrease therefore give the complete chain
\begin{align*}
 \mathcal L\!\left(\bm\theta^{(k+1)}\right)
 &=c\,\Phi\!\left(\bm\theta_{\mathrm{poi}}^{[j_k]}\right)+C \\
 &\leq c\,\Phi\!\left(
 \bm\theta_{\mathrm{poi}}^{(k+\frac12)}\right)+C \\
 &=\mathcal L\!\left(\bm\theta^{(k+\frac12)}\right)
 \leq \mathcal L\!\left(\bm\theta^{(k)}\right).
\end{align*}
This is the stated cycle-wise descent.
\end{proof}

The compatibility condition holds when the omitted residual terms are
constant with respect to the frozen-block selective problem, as in the
experiments below. The theorem does not assert convergence to a global
minimizer, nor does it apply directly when independently resampled stochastic
batches define successive objective values.

\begin{algorithm}[H]
	\caption{Hybrid Joint--Selective Optimization (HJSO)}
	\label{alg:hybrid_fo_lm}
	\begin{algorithmic}[1]
		
		\Require Initial parameters
		$\bm{\theta}^{(0)}
		=
		\begin{bmatrix}
			\bm{\theta}_{\mathrm r}^{(0)}\\
			\bm{\theta}_{\mathrm{poi}}^{(0)}
		\end{bmatrix}$;
		number of outer cycles $T$;
		maximum first-order iterations $T_{\mathrm{FO}}$;
		maximum LM iterations $T_{\mathrm{LM}}$
		
		\Ensure Optimized parameters $\bm{\theta}^{*}$
		
		\For{$k=0,1,\ldots,T-1$}
		
		\Statex
		\State \textbf{Joint phase: optimize all parameters}
		
		\State
		$\bm{\vartheta}^{[0]}
		\gets
		\bm{\theta}^{(k)}$
		
		\For{$j=0,1,\ldots,T_{\mathrm{FO}}-1$}
		
		\State
		$\bm{\vartheta}^{[j+1]}
		\gets
		\operatorname{FOUpdate}
		\left(
		\bm{\vartheta}^{[j]}
		\right)$
		
		\If{first-order stopping criterion is satisfied}
		\State \textbf{break}
		\EndIf
		
		\EndFor
		
		\State
		$\bm{\theta}^{(k+\frac12)}
		\gets
		\bm{\vartheta}^{[j+1]}$
		
		\State Split
		$\bm{\theta}^{(k+\frac12)}
		=
		\begin{bmatrix}
			\bm{\theta}_{\mathrm r}^{(k+\frac12)}\\
			\bm{\theta}_{\mathrm{poi}}^{(k+\frac12)}
		\end{bmatrix}$
		
		\Statex
		\State \textbf{Selective phase: refine only the POIs}
		
		\State Freeze
		$\overline{\bm{\theta}}_{\mathrm r}
		\gets
		\bm{\theta}_{\mathrm r}^{(k+\frac12)}$
		
		\State
		$\bm{\theta}_{\mathrm{poi}}^{[0]}
		\gets
		\bm{\theta}_{\mathrm{poi}}^{(k+\frac12)}$
		
		\For{$j=0,1,\ldots,T_{\mathrm{LM}}-1$}
		
		\State
		$\bm{\theta}_{\mathrm{poi}}^{[j+1]}
		\gets
		\operatorname{LMUpdate}
		\left(
		\overline{\bm{\theta}}_{\mathrm r},
		\bm{\theta}_{\mathrm{poi}}^{[j]}
		\right)$
		
		\If{LM stopping criterion is satisfied}
		\State \textbf{break}
		\EndIf
		
		\EndFor
		
		
		\Statex
		\State \textbf{Recombination}
		
		\State
		$\bm{\theta}^{(k+1)}
		\gets
		\begin{bmatrix}
			\bm{\theta}_{\mathrm r}^{(k+\frac12)}\\
			\bm{\theta}_{\mathrm{poi}}^{[j+1]}
		\end{bmatrix}$
		
		\EndFor
		
		\State \Return
		$\bm{\theta}^{*}
		\gets
		\bm{\theta}^{(T)}$
		
	\end{algorithmic}
\end{algorithm}

\section{Experimental Results}
\label{sec4}

The proposed HJSO method is evaluated on three representative numerical
problems with distinct characteristics: a matrix eigenvalue problem, an
inverse Bratu problem solved using physics-informed neural networks (PINNs)
\cite{shahab2026physics}, and a 100-dimensional nonlinear Black--Scholes
problem solved using the DeepBSDE framework \cite{han2018solving}. These
examples are chosen to test the framework in different settings: a simple
low-dimensional spectral parameter, a nonlinear inverse problem with a
learned state representation, and a high-dimensional stochastic PDE
approximation.

For each problem, HJSO is compared with the corresponding standard
\emph{joint optimization} (JO) baseline under the same problem formulation,
initialization, and first-order optimizer. In JO, all trainable parameters
are updated jointly without the selective POI refinement phase. This setup is
intended to isolate the incremental effect of the reduced-space selective
update while keeping the rest of the optimization problem unchanged.

Performance is measured in terms of convergence speed, final solution or
parameter accuracy, loss reduction, and wall-clock time. Since each HJSO
outer cycle includes both a full-space joint phase and an additional
reduced-space LM step, wall-clock time is used for the comparison rather than
iteration count alone. The optimization and problem-specific settings are
summarized in Tables~\ref{tab:optimization_settings}
and~\ref{tab:problem_settings}.

The comparisons are designed to test the mechanism of selective refinement,
not to claim universal superiority over all optimization methods. JO isolates
the benefit of adding the selective phase, but stronger full-space,
block-structured, or problem-specific solvers may lead to different rankings.
Likewise, wall-clock times depend on the implementation and hardware. The
reported timings are therefore most informative as within-example comparisons
under the present settings.

In the implementation, the selective LM refinement in HJSO is performed
using standard LM solvers. For the eigenvalue and inverse Bratu problems,
MATLAB's \texttt{fsolve} with the Levenberg--Marquardt algorithm is used,
whereas SciPy's \texttt{least\_squares} with \texttt{method='lm'} is used for
the DeepBSDE problem. The Jacobian approximation, damping adjustment, step
acceptance, and internal stopping criteria are therefore delegated to the
corresponding numerical solver. 

\begin{table}[h]
	\scriptsize
	\centering
	\caption{Optimization settings used for JO and HJSO in the numerical experiments.}
	\label{tab:optimization_settings}
	\begin{tabular}{lccc}
		\hline
		Setting
		& Eigenvalue Problem
		& Inverse Bratu PINNs
		& Black--Scholes DeepBSDE \\
		\hline
		
		Outer cycles, $T$
		& 200
		& 400
		& 100 \\
		
		Joint-phase solver
		& MATLAB \texttt{fminunc}
		& MATLAB \texttt{fminunc}
		& TensorFlow \texttt{Adam} \\
		
		Joint-phase algorithm
		& Steepest descent
		& Steepest descent
		& Adam \\
		
		Joint-phase iterations, $T_{\mathrm{FO}}$
		& 50
		& 50
		& 100 \\
		
		Joint-phase step tolerance
		& $10^{-12}$
		& $10^{-12}$
		& -- \\
		
		Selective-phase solver
		& MATLAB \texttt{fsolve}
		& MATLAB \texttt{fsolve}
		& SciPy \texttt{least\_squares} \\
		
		Selective-phase algorithm
		& Levenberg--Marquardt
		& Levenberg--Marquardt
		& Levenberg--Marquardt \\
		
		Selective-phase iterations, $T_{\mathrm{LM}}$
		& 10
		& 10
		& 10 \\
		
		Selective-phase step tolerance
		& $10^{-12}$
		& $10^{-12}$
		& $10^{-6}$ \\
		
		POI dimension, $q$
		& 1
		& 2
		& 1 \\
		\hline
	\end{tabular}
\end{table}

\begin{table}[h]
	\scriptsize
	\centering
	\caption{Problem-specific settings used in the numerical experiments.}
	\label{tab:problem_settings}
	\begin{tabular}{llll}
		\hline
		Setting
		& Eigenvalue Problem
		& Inverse Bratu PINNs
		& Black--Scholes DeepBSDE \\
		\hline
		
		Problem dimension
		& $n=200$
		& One-dimensional
		& $d=100$ \\
		
		Parameter(s) of interest
		& $\lambda$
		& $(\lambda_1,\lambda_2)$
		& $u_0$ \\
		
		Initial POI(s)
		& $\lambda^{(0)}=120$
		& $(\lambda_1^{(0)},\lambda_2^{(0)})=(0,0)$
		& $u_0^{(0)}=100$ \\
		
		Remaining trainable parameters
		& Eigenvector $\bm v$
		& NN weights and biases $\bm w$
		& DeepBSDE NN parameters \\
		
		Neural-network architecture
		& --
		& $\mathrm{NN}(1,20,20,1)$
		& Hidden layers $(110,110)$ \\
		
		Training/collocation points
		& --
		& $99$
		& Batch size $64$ \\
		
		Fixed LM/monitoring samples
		& --
		& --
		& $256$ \\
		
		Time horizon
		& --
		& --
		& $T_{\mathrm{BS}}=1$ \\
		
		Time intervals
		& --
		& --
		& $40$ \\
		
		Data weighting parameter
		& --
		& $\alpha=100$
		& -- \\
		
		Initial remaining parameters
		& Normalized all-ones vector
		& Glorot-type initialization
		& DeepBSDE initialization \\
		\hline
	\end{tabular}
\end{table}

\subsection{Eigenvalue Problem}
\label{subsec:eigenvalue_problem}

The first numerical experiment considers the matrix eigenvalue problem. Given
a matrix $A\in\mathbb{R}^{n\times n}$, the objective is to determine a scalar
$\lambda\in\mathbb{R}$ and a nonzero vector
$\bm v\in\mathbb{R}^{n}$ satisfying
\begin{equation}
	A\bm v=\lambda\bm v,
	\label{eq:eigenval_problem2}
\end{equation}
where $\lambda$ denotes an eigenvalue and $\bm v$ is its corresponding
eigenvector.

In this experiment, a Lehmer matrix of dimension $n=200$ is employed. Its
entries are defined by
\begin{equation}
	A_{ij}
	=
	\frac{\min(i,j)}{\max(i,j)},
	\qquad
	i,j=1,2,\ldots,n.
	\label{eq:lehmer_matrix}
\end{equation}
The Lehmer matrix is symmetric and positive definite, therefore all of its
eigenvalues are real and positive. The objective is to determine its largest
eigenvalue, whose reference value is
$\lambda^{*}\approx109.2516$, together with the corresponding eigenvector.

The eigenvalue problem is formulated as an optimization problem in which the
eigenvalue and eigenvector are estimated simultaneously and iteratively from
prescribed initial values.
For the joint phase, the objective function is defined as
\begin{equation}
	\mathcal L(\lambda,\bm v)
	=
	\frac{1}{2}
	\left\|
	A\bm v-\lambda\bm v
	\right\|_2^2.
	\label{eq:eigen_loss}
\end{equation}
The initial eigenvalue is set to
\begin{equation}
	\lambda^{(0)}=120,
\end{equation}
while the initial eigenvector is chosen as the normalized all-ones vector,
\begin{equation}
	\bm v^{(0)}
	=
	\frac{1}{\sqrt{n}}
	\begin{bmatrix}
		1 & 1 & \cdots & 1
	\end{bmatrix}^{T}.
	\label{eq:eigen_initial_vector}
\end{equation}
To prevent convergence toward the trivial solution $\bm v=\bm 0$, the
eigenvector is normalized according to
\begin{equation}
	\overline{\bm v}
	=
	\frac{\bm v}{\|\bm v\|_2}.
	\label{eq:eigenvector_renormalization}
\end{equation}
This normalization is applied throughout the numerical procedure.

Following the parameter decomposition introduced in
Section~\ref{subsec:hybrid_method}, the parameter blocks are defined as
\begin{equation}
	\bm{\theta}_{\mathrm r}
	=
	\bm v,
	\qquad
	\bm{\theta}_{\mathrm{poi}}
	=
	\lambda.
	\label{eq:eigen_parameter_partition}
\end{equation}
During the joint phase, $\lambda$ and $\bm v$ are optimized simultaneously.
After this phase, the resulting eigenvector is normalized and held fixed,
while LM is applied exclusively to the POI $\lambda$ in the selective phase.

For a fixed normalized eigenvector $\overline{\bm v}$, the reduced problem
solved during the selective phase is
\begin{equation}
	\min_{\lambda}
	\;
	\Phi(\lambda)
	=
	\frac{1}{2}
	\left\|
	\bm r(\lambda,\overline{\bm v})
	\right\|_2^2
	=
	\frac{1}{2}
	\left\|
	A\overline{\bm v}-\lambda\overline{\bm v}
	\right\|_2^2.
	\label{eq:eigen_reduced_problem}
\end{equation}
The Jacobian of the residual with respect to the scalar POI is
\begin{equation}
	\bm J_{\mathrm{poi}}
	=
	\frac{\partial\bm r}
	{\partial\lambda}
	=
	-\overline{\bm v}.
	\label{eq:eigen_special_jacobian}
\end{equation}

This example has an additional structure that makes the selective update
particularly transparent. Because $\|\overline{\bm v}\|_2=1$, the exact
minimizer of Eq.~\eqref{eq:eigen_reduced_problem} is the Rayleigh quotient
\begin{equation}
 \lambda_{\mathrm{sel}}
 =\overline{\bm v}^{T}A\overline{\bm v}.
 \label{eq:eigen_rayleigh_update}
\end{equation}
Thus, an undamped Gauss--Newton step obtains the reduced minimizer in one
iteration, and LM approaches this update as its damping decreases. The
eigenvalue experiment is consequently an illustrative limiting case of HJSO,
not evidence that an iterative LM solve is preferable to the available
closed-form update. In addition, the residual objective vanishes at every
normalized eigenpair; selection of the largest eigenpair in this experiment
depends on the stated initialization and is not guaranteed by the residual
objective alone.

\begin{figure}[h]
	\centering
	
	\begin{subfigure}[b]{0.49\textwidth}
		\centering
		\includegraphics[width=0.9\textwidth]{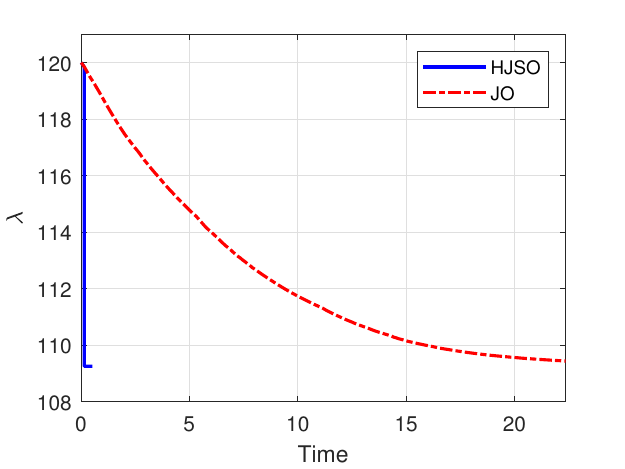}
		\caption{}
		\label{fig_eigenvalue_loss}
	\end{subfigure}
	\begin{subfigure}[b]{0.49\textwidth}
		\centering
		\includegraphics[width=0.9\textwidth]{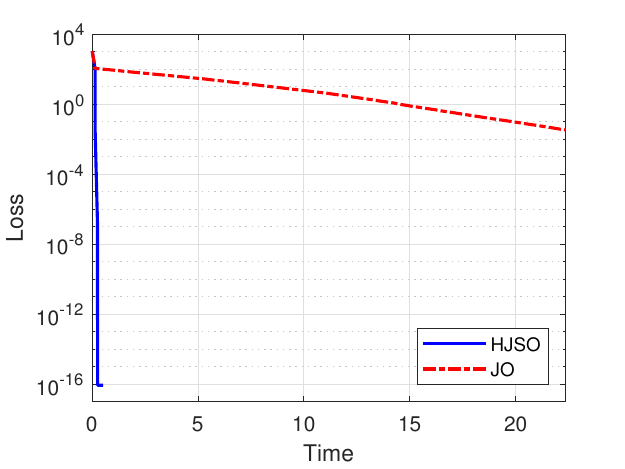}
		\caption{}
		\label{fig_eigenvalue_lambda}
	\end{subfigure}
	
	\begin{subfigure}[b]{0.49\textwidth}
		\centering
		\includegraphics[width=0.9\textwidth]{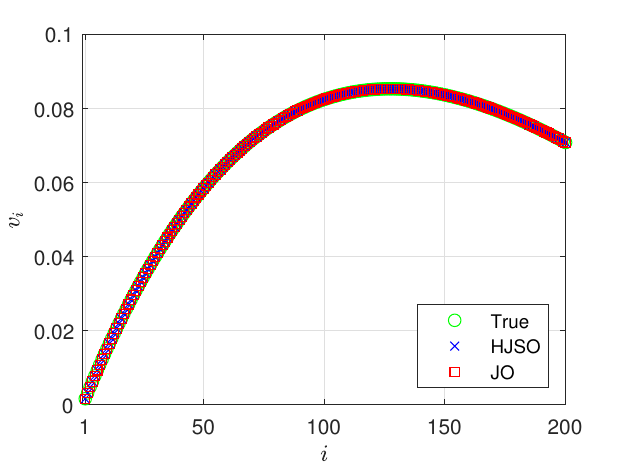}
		\caption{}
		\label{fig_eigenvalue_solution}
	\end{subfigure}
	\begin{subfigure}[b]{0.49\textwidth}
		\centering
		\includegraphics[width=0.9\textwidth]{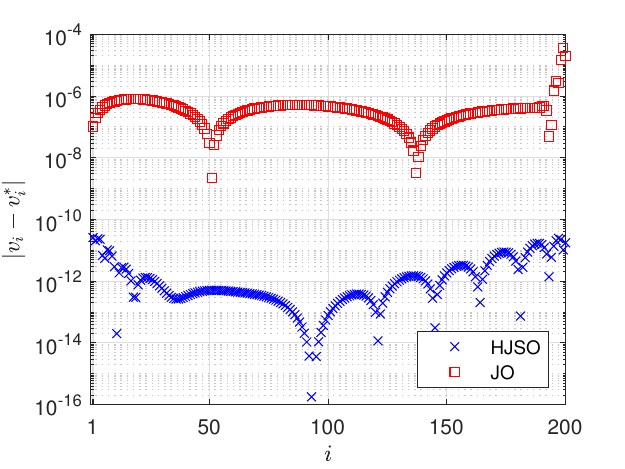}
		\caption{}
		\label{fig_eigenvalue_error}
	\end{subfigure}
	
	\caption{Numerical results for the largest eigenpair of the
		$200\times200$ Lehmer matrix:
		(a) loss versus computational time;
		(b) estimated largest eigenvalue versus computational time;
		(c) reference and computed eigenvectors; and
		(d) pointwise absolute error of the computed eigenvector.}
	\label{fig_eigenvalue}
\end{figure}

Figure~\ref{fig_eigenvalue} presents the numerical comparison between HJSO
and JO. As shown in Figure~\ref{fig_eigenvalue_loss}, HJSO rapidly reduces the
loss and reaches a low-loss regime considerably earlier than JO, whose loss
decreases more gradually over the reported computational interval. This
behavior indicates that selective LM refinement of the eigenvalue accelerates
the reduction of the eigenpair residual.
Figure~\ref{fig_eigenvalue_lambda} shows the convergence of the estimated
eigenvalue. Both methods start from $\lambda^{(0)}=120$ and approach the
reference value $\lambda^{*}\approx109.2516$. However, HJSO approaches the
reference value substantially earlier in computational time than JO.
Figures~\ref{fig_eigenvalue_solution} and~\ref{fig_eigenvalue_error} compare
the corresponding eigenvectors. The computed eigenvectors reproduce the
reference profile, while the pointwise errors provide a more detailed
assessment of their numerical accuracy.

\begin{table}[H]
	\footnotesize
	\centering
	\caption{Comparison of JO and HJSO for the largest eigenvalue of the
		$200\times200$ Lehmer matrix.}
	\label{tab:eigenvalue_comparison}
	\begin{tabular}{cccccc}
		\hline
		Method & Loss & $\lambda$ & APE $\lambda$ & Time (s) &
		Time to 1\% APE (s) \\ \hline
		HJSO & $8.5\times10^{-17}$ & 109.2516 & $8.4\times10^{-13}$ & 0.54 & 0.13 \\
		JO   & $3.4\times10^{-2}$  & 109.4358 & $1.6\times10^{-1}$  & 22   & 14   \\ \hline
	\end{tabular}
\end{table}

The final numerical results are summarized in
Table~\ref{tab:eigenvalue_comparison}, including the absolute percentage error
(APE) of the estimated eigenvalue and the computational time required to reach
an APE of $1\%$. HJSO obtains an eigenvalue of $109.2516$ with an APE of
$8.4\times10^{-13}\%$, whereas JO obtains $109.4358$ with an APE of
$1.6\times10^{-1}\%$. HJSO reaches an APE of $1\%$ in approximately
$0.13$~s, compared with $14$~s for JO.

These results illustrate the effect of selective reduced-space refinement of
the scalar POI. The joint phase updates the eigenvalue and eigenvector
simultaneously, whereas the selective phase holds the normalized eigenvector
fixed and moves the eigenvalue toward the Rayleigh quotient in
Eq.~\eqref{eq:eigen_rayleigh_update}. The large improvement relative to JO in
this example should therefore be interpreted in light of this favorable
closed-form reduced structure.

\subsection{Bratu Equation in an Inverse-Problem PINN Framework}
\label{subsec:bratu_inverse}

The second numerical experiment considers an inverse problem associated with a
two-parameter form of the Bratu equation \cite{shahab2026physics},
\begin{equation}
	\frac{d^2u}{dx^2}
	+
	\lambda_1\exp\left(\lambda_2 u\right)
	=
	0,
	\qquad
	x\in(0,1),
	\label{eq:bratu_general}
\end{equation}
subject to the homogeneous Dirichlet boundary conditions
\begin{equation}
	u(0)=u(1)=0.
	\label{eq:bratu_bc}
\end{equation}
Here, $u(x)$ denotes the unknown solution, while $\lambda_1$ and
$\lambda_2$ are unknown parameters to be identified.

The reference parameters are set to
\begin{equation}
	\lambda_1^{*}=2,
	\qquad
	\lambda_2^{*}=1,
	\label{eq:bratu_true_parameters}
\end{equation}
with the corresponding analytical solution \cite{shahab2025finite, shahab2024neural}
\begin{equation}
	u^{*}(x)
	=
	2
	\log
	\left[
	\frac{\cosh(\vartheta)}
	{\cosh\left(\vartheta(1-2x)\right)}
	\right],
	\label{eq:bratu_true_solution}
\end{equation}
where
\begin{equation}
	\vartheta\approx0.589387763469351.
\end{equation}

To solve the inverse problem using PINNs, the computational domain is
discretized at the interior points
\begin{equation}
	x_i=0.01i,
	\qquad
	i=1,2,\ldots,99,
	\label{eq:bratu_points}
\end{equation}
giving $N_f=N_u=99$ collocation and training points. A fully connected
feedforward neural network with architecture $\mathrm{NN}(1,20,20,1)$ is
employed to approximate the solution \cite{shahab2026physics}. Let
$\tilde{u}(x;\bm w)$ denote the raw network output, where $\bm w$ contains the
network weights and biases.
The boundary conditions are imposed exactly through
\begin{equation}
	u(x;\bm w)
	=
	x(1-x)\tilde{u}(x;\bm w),
	\label{eq:bratu_hard_bc}
\end{equation}
which automatically satisfies $u(0)=u(1)=0$.

The joint-phase loss is defined as
\begin{equation}
	\mathcal L
	=
	\mathrm{MSE}_{f}
	+
	\alpha^2\mathrm{MSE}_{u},
	\label{eq:bratu_loss}
\end{equation}
where
\begin{align}
	\mathrm{MSE}_{f}
	&=
	\frac{1}{N_f}
	\sum_{i=1}^{N_f}
	\left(
	u_{xx}(x_i)
	+
	\lambda_1
	\exp\left(\lambda_2u(x_i)\right)
	\right)^2,
	\label{eq:bratu_msef}
	\\
	\mathrm{MSE}_{u}
	&=
	\frac{1}{N_u}
	\sum_{i=1}^{N_u}
	\left(
	u(x_i)-u^{*}(x_i)
	\right)^2.
	\label{eq:bratu_mseu}
\end{align}
The network weights are initialized using a Glorot-type uniform
initialization, while the biases are initialized to zero. The POIs are
initialized as
\begin{equation}
	\bm{\theta}_{\mathrm{poi}}^{(0)}
	=
	\begin{bmatrix}
		0\\
		0
	\end{bmatrix}.
	\label{eq:bratu_lambda_initial}
\end{equation}

Following the parameter decomposition introduced in
Section~\ref{subsec:hybrid_method}, the parameter blocks are defined as
\begin{equation}
	\bm{\theta}_{\mathrm r}
	=
	\bm w,
	\qquad
	\bm{\theta}_{\mathrm{poi}}
	=
	\begin{bmatrix}
		\lambda_1\\
		\lambda_2
	\end{bmatrix}.
	\label{eq:bratu_parameter_partition}
\end{equation}
During the joint phase, $\bm w$, $\lambda_1$, and $\lambda_2$ are optimized
simultaneously. After this phase, the neural-network parameters are held fixed,
while LM is applied only to the coupled POI block
$(\lambda_1,\lambda_2)^T$.

For fixed neural-network parameters
$\overline{\bm{\theta}}_{\mathrm r}=\overline{\bm w}$, the reduced problem
solved during the selective phase is
\begin{equation}
	\min_{\lambda_1,\lambda_2}
	\;
	\Phi(\lambda_1,\lambda_2)
	=
	\frac{1}{2}
	\left\|
	\bm r_f
	\left(
	\overline{\bm w},
	\lambda_1,
	\lambda_2
	\right)
	\right\|_2^2,
	\label{eq:bratu_reduced_problem}
\end{equation}
where the physics residual vector is
\begin{equation}
	\bm r_f
	=
	\begin{bmatrix}
		r_f(x_1)\\
		r_f(x_2)\\
		\vdots\\
		r_f(x_{N_f})
	\end{bmatrix},
	\qquad
	r_f(x_i)
	=
	u_{xx}(x_i;\overline{\bm w})
	+
	\lambda_1
	\exp\left[
	\lambda_2u(x_i;\overline{\bm w})
	\right].
	\label{eq:bratu_residual_vector}
\end{equation}
The data residual does not appear in the reduced problem because, once
$\bm w$ is fixed, it is independent of $\lambda_1$ and $\lambda_2$ and
therefore does not affect the selective POI refinement.
The Jacobian of the physics residual with respect to the POIs is
\begin{equation}
	\bm J_{\mathrm{poi}}
	=
	\begin{bmatrix}
		\displaystyle
		\frac{\partial\bm r_f}{\partial\lambda_1}
		&
		\displaystyle
		\frac{\partial\bm r_f}{\partial\lambda_2}
	\end{bmatrix}
	\in\mathbb{R}^{N_f\times2},
	\label{eq:pinn_special_jacobian}
\end{equation}
with row $i$ given by
\begin{equation}
	\left[
	\exp\left(\lambda_2u(x_i)\right),
	\quad
	\lambda_1u(x_i)
	\exp\left(\lambda_2u(x_i)\right)
	\right].
\end{equation}

\begin{figure}[h]
	\centering
	
	\begin{subfigure}[b]{0.49\textwidth}
		\centering
		\includegraphics[width=0.9\textwidth]{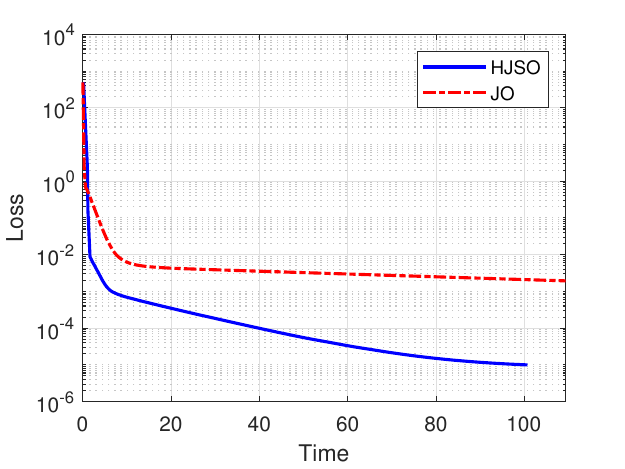}
		\caption{}
		\label{fig_bratu_loss}
	\end{subfigure}
	\begin{subfigure}[b]{0.49\textwidth}
		\centering
		\includegraphics[width=0.9\textwidth]{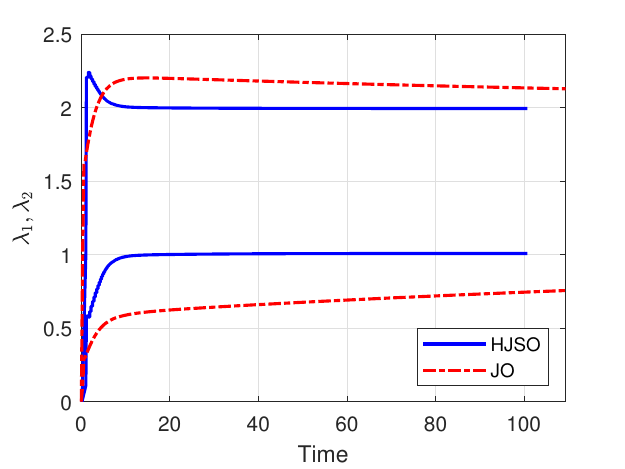}
		\caption{}
		\label{fig_bratu_lambda}
	\end{subfigure}
	
	\begin{subfigure}[b]{0.49\textwidth}
		\centering
		\includegraphics[width=0.9\textwidth]{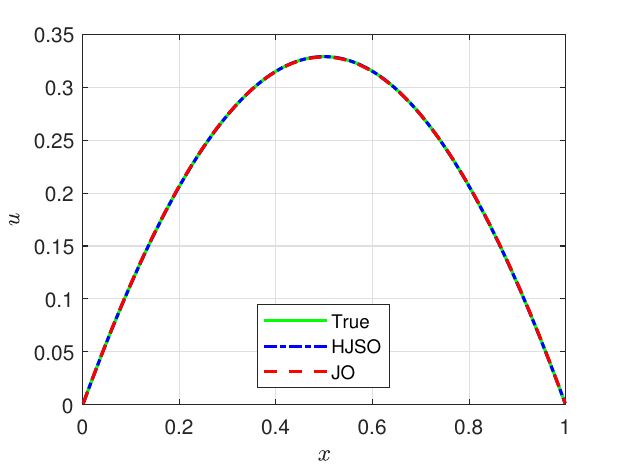}
		\caption{}
		\label{fig_bratu_solution}
	\end{subfigure}
	\begin{subfigure}[b]{0.49\textwidth}
		\centering
		\includegraphics[width=0.9\textwidth]{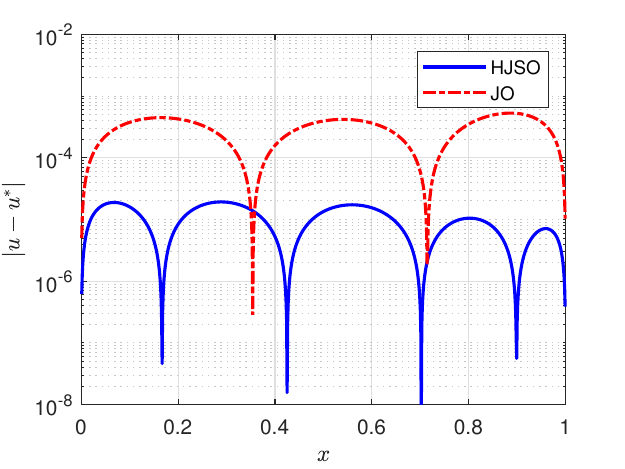}
		\caption{}
		\label{fig_bratu_error}
	\end{subfigure}
	
	\caption{Numerical results for the inverse Bratu problem:
		(a) loss versus computational time;
		(b) identified parameters $\lambda_1$ and $\lambda_2$ versus
		computational time;
		(c) analytical reference and predicted solutions; and
		(d) pointwise absolute error $|u-u^{*}|$.}
	\label{fig_bratu}
\end{figure}

Figure~\ref{fig_bratu} presents the numerical comparison between HJSO and JO.
As shown in Figure~\ref{fig_bratu_loss}, HJSO rapidly reduces the loss and
reaches a substantially lower final value than JO. In contrast, JO exhibits a
rapid initial reduction followed by a considerably slower convergence regime.
This behavior indicates that the selective LM phase provides additional
refinement after the joint optimization.
Figure~\ref{fig_bratu_lambda} shows the convergence of the identified
parameters. HJSO rapidly approaches the reference values
$(\lambda_1^{*},\lambda_2^{*})=(2,1)$ and remains close to them, whereas the
parameters obtained using JO remain noticeably separated from the reference
values over the reported computational interval.
Figures~\ref{fig_bratu_solution} and~\ref{fig_bratu_error} compare the
corresponding solution approximations. Both methods reproduce the overall
reference solution profile, but HJSO provides a more accurate pointwise
approximation over most of the domain. The results also show that a visually
accurate state approximation does not necessarily imply equally accurate
identification of the unknown parameters.

\begin{table}[H]
	\footnotesize
	\centering
	\caption{Comparison of JO and HJSO for the inverse Bratu problem.}
	\label{tab:bratu_comparison}
	\begin{tabular}{cccccccc}
		\hline
		Method & Loss & $\lambda_1$ & $\lambda_2$ &
		APE $\lambda_1$ & APE $\lambda_2$ & Time (s) & Time to 1\% APE (s) \\ \hline
		HJSO & $9.9\times10^{-6}$ & 1.9955 & 1.0082 &
		0.2236 & 0.8250 & 99 & 9.1 \\
		JO & $1.9\times10^{-3}$ & 2.1295 & 0.7559 &
		6.4772 & 24.4068 & 109 & -- \\ \hline
	\end{tabular}
\end{table}

The final numerical results are summarized in
Table~\ref{tab:bratu_comparison}, including the APEs of the identified
parameters and the computational time required for both POIs to reach an APE
below $1\%$. HJSO identifies $\lambda_1=1.9955$ and $\lambda_2=1.0082$,
with APEs of $0.2236\%$ and $0.8250\%$, respectively. In comparison, JO
obtains $\lambda_1=2.1295$ and $\lambda_2=0.7559$, with APEs of $6.4772\%$
and $24.4068\%$, respectively. HJSO reaches an APE below $1\%$ for both POIs
in $9.1678$~s, whereas JO does not reach this threshold during the reported
computational interval. HJSO also achieves a lower final loss
$9.9\times10^{-6}$ compared with $1.9\times10^{-3}$ for JO.

These results show how selective refinement behaves for a two-dimensional POI
block. The joint phase updates the neural-network
parameters and both unknown coefficients simultaneously, whereas the
selective phase fixes the network parameters and directly refines
$(\lambda_1,\lambda_2)^T$. This experiment therefore extends the HJSO
mechanism from a scalar POI to multiple coupled POIs.

The experiment uses dense, noise-free synthetic observations at the same
locations as the collocation points. It is designed to isolate optimization
behavior under controlled conditions, rather than to establish robustness to
sparse or noisy inverse data. In particular, the reported parameter errors
should not be extrapolated to observational settings with measurement noise
or model discrepancy.

\subsection{High-Dimensional Black--Scholes Equation Using the DeepBSDE Method}
\label{subsec:deepbsde_blackscholes}

The third numerical experiment considers a high-dimensional nonlinear
Black--Scholes equation with default risk, introduced as a benchmark problem
for the Deep Backward Stochastic Differential Equation (DeepBSDE) method
\cite{han2018solving}. The model incorporates the possibility of default of
the claim issuer, resulting in a nonlinear pricing equation.
Let
\begin{equation}
	u:[0,T_{\mathrm{BS}}]\times\mathbb{R}^{d}
	\rightarrow\mathbb{R}
\end{equation}
denote the value of a European contingent claim depending on $d$ underlying
assets. For the default-risk model considered here, the underlying assets
follow independent geometric Brownian motions,
\begin{equation}
	dX_t^{i}
	=
	\bar{\mu}X_t^{i}\,dt
	+
	\bar{\sigma}X_t^{i}\,dW_t^{i},
	\qquad
	i=1,2,\ldots,d,
	\label{eq:blackscholes_forward_sde}
\end{equation}
and the resulting nonlinear Black--Scholes equation is
\begin{equation}
	\frac{\partial u}{\partial t}
	+
	\bar{\mu}\bm{x}\cdot\nabla_{\bm{x}}u
	+
	\frac{\bar{\sigma}^{2}}{2}
	\sum_{i=1}^{d}
	x_i^{2}
	\frac{\partial^{2}u}{\partial x_i^{2}}
	-
	(1-\delta)Q(u)u
	-
	Ru
	=
	0,
	\label{eq:blackscholes_default}
\end{equation}
where $\delta$ is the recovery rate, $R$ is the risk-free interest rate, and
the default intensity is defined by
\begin{equation}
	Q(u)
	=
	\begin{cases}
		\gamma^{h},
		& u<v^{h},
		\\[1mm]
		\displaystyle
		\gamma^{h}
		+
		\frac{\gamma^{l}-\gamma^{h}}
		{v^{l}-v^{h}}
		(u-v^{h}),
		& v^{h}\leq u<v^{l},
		\\[3mm]
		\gamma^{l},
		& u\geq v^{l}.
	\end{cases}
	\label{eq:default_intensity}
\end{equation}

Following the benchmark in \cite{han2018solving}, the dimension is $d=100$,
with
\begin{equation}
	T_{\mathrm{BS}}=1,
	\qquad
	\delta=\frac{2}{3},
	\qquad
	R=0.02,
	\qquad
	\bar{\mu}=0.02,
	\qquad
	\bar{\sigma}=0.2,
	\label{eq:blackscholes_parameters1}
\end{equation}
and
\begin{equation}
	v^{h}=50,
	\qquad
	v^{l}=70,
	\qquad
	\gamma^{h}=0.2,
	\qquad
	\gamma^{l}=0.02.
	\label{eq:blackscholes_parameters2}
\end{equation}
The terminal payoff is
\begin{equation}
	g(\bm{x})
	=
	\min_{1\leq i\leq d}x_i,
	\label{eq:blackscholes_terminal_payoff}
\end{equation}
and the initial state is
\begin{equation}
	\bm{X}_0
	=
	\begin{bmatrix}
		100 & 100 & \cdots & 100
	\end{bmatrix}^{T}
	\in\mathbb{R}^{100}.
	\label{eq:blackscholes_initial_state}
\end{equation}
The reference value reported for this benchmark is
$u^{*}(0,\bm{X}_0)\approx57.3$ \cite{han2018solving}. Since no closed-form
solution is available, this benchmark value is used as the reference. Its
reported precision should be kept in mind when interpreting percentage errors
below the one-percent level.

The DeepBSDE method reformulates the PDE through the corresponding
forward--backward stochastic differential equations. Along a stochastic
trajectory, define
\begin{equation}
	u_t
	=
	u(t,\bm{X}_t),
	\qquad
	\bm{Z}_t
	=
	\bm{\sigma}(t,\bm{X}_t)^{T}
	\nabla_{\bm{x}}u(t,\bm{X}_t),
	\label{eq:YZ_definition}
\end{equation}
where $u_t$ denotes the solution evaluated along the forward stochastic
process. The corresponding backward process satisfies
\begin{equation}
	du_t
	=
	-
	f(t,\bm{X}_t,u_t,\bm{Z}_t)\,dt
	+
	\bm{Z}_t^{T}d\bm{W}_t,
	\label{eq:backward_sde}
\end{equation}
where $f$ denotes the nonlinear generator associated with the PDE, subject to
the terminal condition
\begin{equation}
	u_{T_{\mathrm{BS}}}
	=
	g(\bm{X}_{T_{\mathrm{BS}}}).
	\label{eq:bsde_terminal}
\end{equation}
The interval $[0,T_{\mathrm{BS}}]$ is divided into $N=40$ subintervals. 
The forward and backward processes are
approximated by
\begin{equation}
	\bm{X}_{n+1}
	=
	\bm{X}_{n}
	+
	\bm{\mu}(t_n,\bm{X}_n)\Delta t_n
	+
	\bm{\sigma}(t_n,\bm{X}_n)
	\Delta\bm{W}_n,
	\label{eq:deepbsde_forward_discrete}
\end{equation}
and
\begin{equation}
	u_{n+1}
	=
	u_n
	-
	f(t_n,\bm{X}_n,u_n,\bm{Z}_n)\Delta t_n
	+
	\bm{Z}_n^{T}\Delta\bm{W}_n,
	\label{eq:deepbsde_backward_discrete}
\end{equation}
respectively.

The unknown process $\bm Z_n$ is approximated using neural networks,
\begin{equation}
	\bm{Z}_n
	\approx
	\mathcal{N}_n
	\left(
	\bm{X}_n;\bm{w}_n
	\right),
	\label{eq:deepbsde_z_network}
\end{equation}
with two hidden layers of 110 neurons each. Starting from the unknown initial
value $u_0$, the system is propagated to the terminal time. The initial value
$u_0$ and the neural-network parameters are then determined by minimizing the
terminal discrepancy.
The joint-phase loss is defined as
\begin{equation}
	\mathcal L
	=
	\frac{1}{M}
	\sum_{m=1}^{M}
	\left(
	u_N^{(m)}
	-
	g(\bm{X}_N^{(m)})
	\right)^2
	=
	\frac{1}{M}
	\|\bm r\|_2^2,
	\label{eq:deepbsde_loss}
\end{equation}
where $M$ denotes the number of simulated trajectories and the terminal
residual for the $m$-th trajectory is
\begin{equation}
	r_m
	=
	u_N^{(m)}
	-
	g\left(\bm{X}_N^{(m)}\right),
	\qquad
	m=1,2,\ldots,M.
	\label{eq:deepbsde_residual}
\end{equation}
In the numerical implementation, the joint phase uses stochastic mini-batches
of 64 trajectories, while a fixed set of 256 trajectories is used for the
selective LM refinement and for monitoring the reported loss. To examine
convergence from an initialization far from the reference value, the POI is
initialized as
\begin{equation}
	u_0^{(0)}=100.
	\label{eq:deepbsde_initial_y0}
\end{equation}

Following the parameter decomposition introduced in
Section~\ref{subsec:hybrid_method}, the parameter blocks are defined as
\begin{equation}
	\bm{\theta}_{\mathrm r}
	=
	\bm{\theta}_{\mathrm{nn}},
	\qquad
	\bm{\theta}_{\mathrm{poi}}
	=
	u_0,
	\label{eq:deepbsde_parameter_partition}
\end{equation}
where $\bm{\theta}_{\mathrm{nn}}$ contains the neural-network parameters used
to approximate $\bm Z_n$. The scalar $u_0$ is selected as the POI because it
directly represents the desired solution $u(0,\bm X_0)$.
During the joint phase, $\bm{\theta}_{\mathrm{nn}}$ and $u_0$ are optimized
simultaneously. After this phase, the neural-network parameters are held fixed,
while LM is applied exclusively to $u_0$. The optimization settings are
summarized in Table~\ref{tab:optimization_settings}.

For fixed neural-network parameters
$\overline{\bm{\theta}}_{\mathrm r}$, the reduced problem solved during the
selective phase is
\begin{equation}
	\min_{u_0}
	\;
	\Phi(u_0)
	=
	\frac{1}{2}
	\left\|
	\bm r
	\left(
	\overline{\bm{\theta}}_{\mathrm r},
	u_0
	\right)
	\right\|_2^2,
	\label{eq:deepbsde_reduced_problem}
\end{equation}
where
\begin{equation}
	\bm r
	\left(
	\overline{\bm{\theta}}_{\mathrm r},
	u_0
	\right)
	=
	\begin{bmatrix}
		u_N^{(1)}-g(\bm X_N^{(1)})\\
		u_N^{(2)}-g(\bm X_N^{(2)})\\
		\vdots\\
		u_N^{(M)}-g(\bm X_N^{(M)})
	\end{bmatrix}
	\in\mathbb{R}^{M}.
	\label{eq:deepbsde_selective_residual}
\end{equation}
Here, the neural-network parameters are fixed at
$\overline{\bm{\theta}}_{\mathrm r}$, so the terminal residual depends only
on the scalar POI $u_0$ during the selective refinement.
The corresponding POI Jacobian is
\begin{equation}
	\bm J_{\mathrm{poi}}
	=
	\frac{\partial\bm r}
	{\partial u_0}
	\in\mathbb{R}^{M\times1},
	\label{eq:deepbsde_special_jacobian}
\end{equation}
which is approximated using finite differences in the numerical
implementation. Hence, despite the 100-dimensional stochastic problem and
the high-dimensional neural-network parameter space, the selective LM problem
remains scalar.

\begin{figure}[h]
	\centering
	
	\begin{subfigure}[b]{0.49\textwidth}
		\centering
		\includegraphics[width=0.9\textwidth]{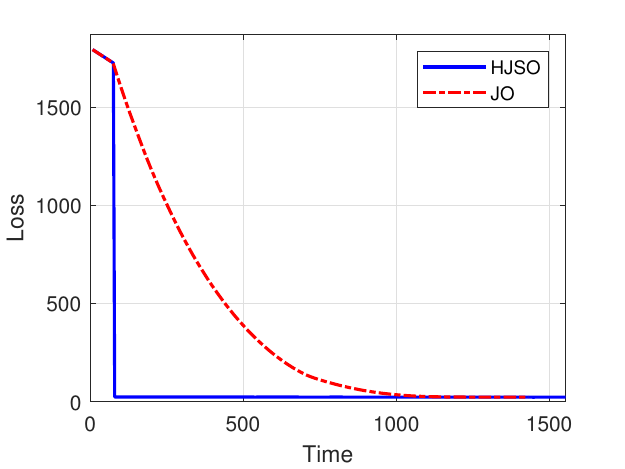}
		\caption{}
		\label{fig_deepbsde_loss}
	\end{subfigure}
	\begin{subfigure}[b]{0.49\textwidth}
		\centering
		\includegraphics[width=0.9\textwidth]{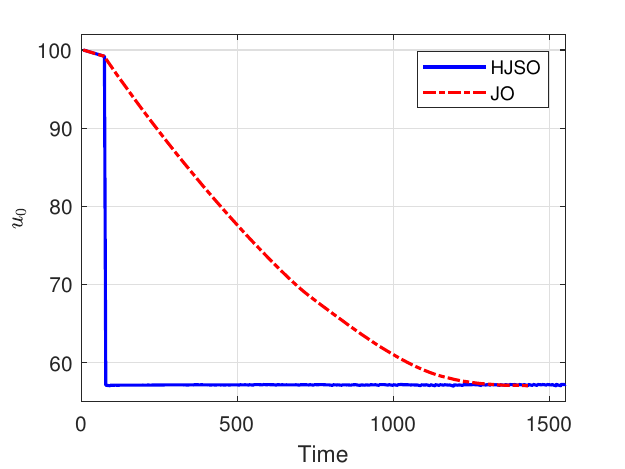}
		\caption{}
		\label{fig_deepbsde_u0}
	\end{subfigure}
	
	\caption{Numerical results for the 100-dimensional nonlinear
		Black--Scholes problem:
		(a) monitored fixed-sample loss versus computational time; and
		(b) estimated initial value
		$u_0=u(0,\bm X_0)$ versus computational time.}
	\label{fig_deepbsde}
\end{figure}

Figure~\ref{fig_deepbsde} presents the numerical comparison between HJSO and
JO. As shown in Figure~\ref{fig_deepbsde_loss}, both HJSO and JO reduce the
monitored loss to similar levels, with HJSO achieving a slightly lower final
value. Because the fixed trajectories are also used during selective LM
refinement, this quantity is a monitoring loss rather than an independent
out-of-sample validation loss.
Figure~\ref{fig_deepbsde_u0} shows the convergence of the estimated initial
value. Both methods start from $u_0^{(0)}=100$ and approach the reference
value $57.3$. However, HJSO moves the POI toward the reference value
substantially earlier in computational time than JO.

\begin{table}[H]
	\footnotesize
	\centering
	\caption{Comparison of JO and HJSO for the 100-dimensional nonlinear
		Black--Scholes problem.}
	\label{tab:black_scholes_comparison}
	\begin{tabular}{cccccc}
		\hline
		Method & Loss & $u_0$ & APE $u_0$ & Time (s) & Time to 1\% APE (s) \\ \hline
		HJSO & 22.6285 & 57.1889 & 0.1938 & 1553 & 78 \\
		JO   & 22.7697 & 57.0165 & 0.4948 & 1432 & 1198 \\ \hline
	\end{tabular}
\end{table}

The final numerical results are summarized in
Table~\ref{tab:black_scholes_comparison}, including the APE of the estimated
initial value and the computational time required to reach an APE of $1\%$.
HJSO obtains $u_0=57.1889$ with an APE of $0.1938\%$, whereas JO obtains
$u_0=57.0165$ with an APE of $0.4948\%$. HJSO reaches an APE of $1\%$ in
$78$~s, compared with $1198$~s for JO. Although the total computational time of HJSO is slightly longer due to the additional LM refinement, it reaches the $1\%$ APE threshold substantially
earlier.

These results illustrate the HJSO mechanism in a high-dimensional stochastic
setting. The joint phase optimizes the neural-network parameters together with
the POI, whereas the selective phase fixes the network parameters and directly
refines the scalar quantity $u_0=u(0,\bm X_0)$. Thus, the curvature-based
refinement remains low-dimensional despite the high dimensionality of the
underlying problem. The experiment supports faster POI refinement for this
fixed sampling configuration, but it does not by itself quantify variation
across independent trajectory sets or random initializations.

%

\section{Discussion and Limitations}
\label{sec:discussion}

The numerical experiments consistently indicate that, when the quantity of
interest is tied to a small parameter block, a reduced-space LM update can
accelerate convergence relative to continued joint first-order updates. This
is the central empirical observation behind HJSO. The method is therefore
best understood as a structured reduced-space refinement strategy rather than
as a universal replacement for standard optimizers.

Several points are important for interpreting the reported results. First,
HJSO does not resolve the nonconvexity of the underlying joint optimization
problem and does not guarantee the selection of a desired solution branch.
Second, the cost of evaluating residuals and Jacobians can still dominate the
computation, even when the LM system itself is small. Third, the present
experiments isolate the value of the selective phase by comparing against a
joint first-order baseline; they are not intended as a broad benchmarking
study against all first-order, quasi-Newton, block-coordinate, or
variable-projection alternatives. Finally, the results correspond to the
specified initializations and sampling configurations, and the PINN and
DeepBSDE examples do not yet establish a full statistical robustness study
across random seeds or data realizations.

These limitations define the intended scope of the method. HJSO is most
appropriate when the POIs are known in advance, their dimension is small,
the governing residual is informative in that reduced subspace, and the full
parameter vector is too large to allow affordable full-space second-order
updates. Under these conditions, the method provides a practical compromise
between scalability and targeted refinement of the quantities that matter most.

\section{Conclusion}
\label{sec:conclusion}

In this work, we proposed a hybrid joint--selective optimization (HJSO)
framework for problems in which the full parameter vector is large, but the
parameter block of primary interest is low-dimensional. The method partitions the
trainable variables into a high-dimensional remaining block,
$\bm{\theta}_{\mathrm r}\in\mathbb{R}^{p}$, and a low-dimensional POI block,
$\bm{\theta}_{\mathrm{poi}}\in\mathbb{R}^{q}$ with $q\ll p$. A first-order
optimization stage updates the full parameter set, and a reduced-space LM
refinement then updates the POIs while holding the remaining variables fixed.
This combination preserves the scalability of the joint optimization step
while exploiting curvature information only in the reduced subspace that
matters most for the target quantity.

The method was evaluated on three representative problems: an eigenvalue
problem for a $200\times200$ Lehmer matrix, an inverse Bratu problem formulated with a
PINN, and a high-dimensional nonlinear Black--Scholes problem solved with the
DeepBSDE method. In all three cases, HJSO reached prescribed POI-error
thresholds faster than the corresponding joint first-order baseline and
produced more accurate final estimates for the target quantities under the
stated settings.

Within the scope stated in Section~\ref{sec:discussion}, the results support
HJSO as a practical reduced-space strategy when a small, prescribed parameter
block directly governs the target output.

Future work may consider larger POI blocks, systematic selection of POIs,
comparisons with quasi-Newton and variable-projection methods, and statistical
studies across initializations and stochastic realizations. These directions
would clarify the range of applicability of the method and help establish
when reduced-space refinement is most beneficial.


%

\section*{Declaration of competing interest} 
The authors declare that they have no known competing financial interests or personal relationships that could have appeared to influence the work reported in this paper.



\section*{Declaration of generative AI and AI-assisted technologies in the writing process}

During the preparation of this work, the authors used ChatGPT to improve the language and readability. 
After using this tool, the authors reviewed and edited the content as needed and take full responsibility for the content of the publication.

\bibliographystyle{elsarticle-num} 
\bibliography{main_references}

%
%

\end{document}